\documentclass[10pt,twoside]{article}
\usepackage{amsfonts}
\usepackage{fancyhdr}
\usepackage{titlesec}
\usepackage{cite}
\usepackage{ifthen}
\usepackage{amssymb}
\usepackage{pifont}
\usepackage{stmaryrd}
\usepackage{setspace}
\usepackage{indentfirst}
\usepackage{amsmath,amssymb,amscd,bbm,amsthm,mathrsfs,dsfont}
\input amssym.def

\titleformat{\section}{\centering\large\bfseries}{\S\arabic{section}}{1em}{}
\newboolean{first}
\setboolean{first}{true}

\newtheorem{theorem}{Theorem}[section]
\newtheorem{lemma}{Lemma}[section]

\newtheorem{definition}{Definition}[section]

\begin{document}

\setlength\abovedisplayskip{2pt}
\setlength\abovedisplayshortskip{0pt}
\setlength\belowdisplayskip{2pt}
\setlength\belowdisplayshortskip{0pt}

\title{\bf \Large Necessary and sufficient conditions for boundedness of commutators of bilinear Hardy-Littlewood maximal function\author{Wang Ding-huai and Zhou Jiang$^*$}\date{}} \maketitle
 \footnote{Received: 2016-**-**.}
 \footnote{MR Subject Classification: 42B20, 42B25, 42B35.} 
 \footnote{Digital Object Identifier(DOI): 10.1007/s11766-013-****-*.}
 \footnote{Supported by the National Natural Science Foundation of China\,(11661075)}
 \footnote{$^*$Corresponding author}
\begin{center}
\begin{minipage}{135mm}

{\bf \small Abstract}.\hskip 2mm {\small Let $\mathcal{M}$ be the bilinear Hardy-Littlewood  maximal function and $\vec{b}=(b,b)$ be a collection of locally integrable functions. In this paper, the authors establish characterizations of the weighted {\rm BMO} space in terms of several different
commutators of bilinear Hardy-Littlewood maximal function, respectively; these commutators include the maximal iterated commutator $\mathcal{M}_{\Pi \vec{b}}$, the maximal linear commutator $\mathcal{M}_{\Sigma\vec{b}}$, the iterated commutator $[\Pi \vec{b},\mathcal{M}]$ and the linear commutator $[\Sigma \vec{b},\mathcal{M}]$.}
\end{minipage}
\end{center}

\thispagestyle{fancyplain} \fancyhead{}
\fancyhead[L]{\textit{Appl. Math. J. Chinese Univ.}\\
2013, 28(*): ***-***} \fancyfoot{} \vskip 10mm

\section{Introduction}

A locally integrable function $f$ is said to belong to \rm{BMO} space if there exists a constant
$C > 0$ such that for any cube $Q\subset \mathbb{R}^n$,
$$\frac{1}{|Q|}\int_{Q}|f(x)-f_{Q}|dx\leq C,$$
where $f_{Q}=\frac{1}{|Q|}\int_{Q}f(x)dx$ and the minimal constant $C$ is defined by $\|f\|_{*}$.

There are a number of classical results that demonstrate ${\rm BMO}$ functions are the right collections to do harmonic analysis on the boundedness of commutators. A well known result of Coifman, Rochberg and Weiss \cite{CRW} states that the commutator
$$[b,T](f)=bT(f)-T(bf)$$
is bounded on some $L^p$, $1<p<\infty$, if and only if $b\in \mathrm{BMO}$, where $T$ be the classical Calder\'{o}n-Zygmund operator. Chanillo \cite{C} proved that if $b\in {\rm BMO}$, the commutator
$$[b,I_{\alpha}]f(x)=b(x)I_{\alpha}f(x)-I_{\alpha}(bf)(x)$$
is bounded from $L^{p}$ to $L^{q}$ with $1<p<n/\alpha$ and $1/q=1/p-\alpha/n$, where $I_{\alpha}$ be a fractional integral operator. Moreover, if $n-\alpha$ is even, the reverse is also valid. A complete characterization of ${\rm BMO}$ via the commutator $[b,I_{\alpha}]$ was shown by Ding \cite{D}. During the past thirty years, the theory was then extended and generalized to several directions. For instance, Bloom \cite{B} investigated the characterization of {\rm BMO} spaces in the weighted setting. In 1991, Garc\'{\i}a-Cuerva, Harboure, Segovia and Torrea \cite{GHS} showed that the maximal commutator
\begin{eqnarray*}
M_{b}(f)(x)=\sup_{Q\ni x}\frac{1}{|Q|}\int_{Q}|b(x)-b(y)||f(y)|dy
\end{eqnarray*}
is bounded on $L^{p}$, $1<p<\infty$, if and only if $b\in {\rm BMO}$. In 2000, Bastero, Milman and Ruiz \cite{BMR} studied the necessary and sufficient conditions for the boundedness of $[b,M]$ on $L^p$ spaces when $1 < p < \infty$. They showed that the commutator of Hardy-Littlewood maximal operator
$$[b, M](f)(x)=b(x)M(f)(x)-M(bf)(x)$$
is bounded on $L^{p}$, $1<p<\infty$, if and only if $b\in {\rm BMO}$ with $b^{-}\in L^{\infty}$, where $b^{-}(x)=-\min\{b(x),0\}$. In 2014, Zhang \cite{Z1} considered the characterization of ${\rm BMO}$ via the commutator of the fractional maximal function on variable exponent Lebesgue spaces.

In the multilinear setting, the boundedness of commutators has been extensively studied already, as in P\'{e}rez and Torres' \cite{PT}, Tang¡¯s \cite{T}, Lerner, Ombrosi, P\'{e}rez, Torres, and Trujillo-Gonz\'{a}lez¡¯s \cite{LOPTT} and Chen and Xue¡¯s \cite{CX}, and P\'{e}rez, Pradolini, Torres, and Trujillo-Gonz\'{a}lez¡¯s \cite{PPTT}. Specially, Chaffee and Torres \cite{CT}, Wang, Pan and Jiang \cite{JWP} and Zhang \cite{Z} contributed the theory of characterization of {\rm BMO} spaces by considering the {\bf linear commutator} of Multilinear operators, respectively. 
In this paper, we will extend Zhang's result to weighted case and we replace the linear commutators by {\bf iterated commutators}.
\vspace{0.3cm}

Our main results as follows.

\begin{theorem}\label{thm1.1}
Let $1< p_{1},p_{2}<\infty, \vec{b}=(b,b), 1/p=1/p_{1}+1/p_{2}$ and $\omega\in A_{1}$. Then the following are equivalent,
\begin{enumerate}
\item [\rm(A1)] $b\in {\rm BMO}(\omega)$;
\item [\rm(A2)] $\mathcal{M}_{\Sigma\vec{b}}$ is bounded from $L^{p_{1}}(\omega)\times L^{p_{2}}(\omega)$ to $L^{p}(\omega^{1-p})$;
\item [\rm(A3)] $\mathcal{M}_{\Pi\vec{b}}$ is bounded from $L^{p_{1}}(\omega)\times L^{p_{2}}(\omega)$ to $L^{p}(\omega^{1-2p})$.
\end{enumerate}
\end{theorem}

\vspace{0.3cm}

\begin{theorem}\label{thm1.2}
Let $1< p_{1},p_{2}<\infty, \vec{b}=(b,b), 1/p=1/p_{1}+1/p_{2}$ and $\omega\in A_{1}$. Then the following are equivalent,
\begin{enumerate}
\item [\rm(B1)] $b\in {\rm BMO}(\omega)$ and $b^{-}/\omega\in L^{\infty}$;
\item [\rm(B2)] $[\Sigma\vec{b},\mathcal{M}]$ is bounded from $L^{p_{1}}(\omega)\times L^{p_{2}}(\omega)$ to $L^{p}(\omega^{1-p})$;
\item [\rm(B3)] $[\Pi\vec{b},\mathcal{M}]$ is bounded from $L^{p_{1}}(\omega)\times L^{p_{2}}(\omega)$ to $L^{p}(\omega^{1-2p})$.
\end{enumerate}
\end{theorem}

\section{Some preliminaries and notations}

In 2009, Lerner, Ombrosi, P\'{e}rez, Torres and Trujillo-Gonz\'{a}lez [12] introduced the following multilinear maximal function that adapts to the multilinear Calder\'{o}n-Zygmund theory. In this paper, we only consider the bilinear case. A similar argument also works for the multilinear cases.

\begin{definition}
For a collection of locally integrable functions $\vec{f}=(f_{1},f_{2})$, the bilinear
maximal function $\mathcal{M}$ is defined by
$$\mathcal{M}(\vec{f})(x)=\sup_{Q\ni x}\prod_{i=1}^{2}\frac{1}{|Q|}\int_{Q}|f_{i}(y_{i})|dy_{i}.$$
\end{definition}

We now give the definitions of the maximal commutators and the commutators related to the bilinear maximal function $\mathcal{M}$.

\begin{definition}
For two collections of locally integrable functions $\vec{f}=(f_{1},f_{2})$ and $\vec{b}=(b_{1},b_{2})$, the maximal linear commutator $\mathcal{M}_{\Sigma\vec{b}}$ is defined by
$$\mathcal{M}_{\Sigma\vec{b}}(\vec{f})(x)=\sum_{i=1}^{2}\mathcal{M}^{(i)}_{b_{i}}(\vec{f})(x),$$
where
$$\mathcal{M}^{(i)}_{b_{i}}(\vec{f})(x)=\sup_{Q\ni x}\frac{1}{|Q|^2}\int_{Q}\int_{Q}|b_{i}(x)-b_{i}(y_{i})|\prod_{j=1}^{2}|f_{j}(y_{j})|dy_{1}dy_{2}.$$

The maximal iterated commutator $\mathcal{M}_{\Pi\vec{b}}$ is defined by
$$\mathcal{M}_{\Pi\vec{b}}(\vec{f})(x)=\sup_{Q\ni x}\frac{1}{|Q|^2}\int_{Q}\int_{Q}\prod_{i=1}^{2}|b_{i}(x)-b_{i}(y_{i})||f_{i}(y_{i})|dy_{1}dy_{2}.$$

The linear commutator of $\mathcal{M}$ is defined by
$$[\Sigma\vec{b},\mathcal{M}](\vec{f})(x)=[b_{1},\mathcal{M}]^{(1)}(\vec{f})(x)+[b_{2},\mathcal{M}]^{(2)}(\vec{f})(x),$$
where
$$[b_{1},\mathcal{M}]^{(1)}(\vec{f})(x)=b_{1}(x)\mathcal{M}(\vec{f})(x)-\mathcal{M}(b_{1}f_{1},f_{2})(x)$$
and
$$[b_{2},\mathcal{M}]^{(2)}(\vec{f})(x)=b_{2}(x)\mathcal{M}(\vec{f})(x)-\mathcal{M}(f_{1},b_{2}f_{2})(x).$$

The iterated commutator of $\mathcal{M}$ is defined by
\begin{eqnarray*}
[\Pi\vec{b},\mathcal{M}](\vec{f})(x)&=&b_{1}(x)b_{2}(x)\mathcal{M}(\vec{f})(x)
-b_{1}(x)\mathcal{M}(f_{1},b_{2}f_{2})(x)\\
&&-b_{2}(x)\mathcal{M}(b_{1}f_{1},f_{2})(x)
+\mathcal{M}(b_{1}f_{1},b_{2}f_{2})(x).
\end{eqnarray*}
\end{definition}

We now recall the definition of $A_{p}$ weight introduced by Muckenhoupt \cite{M}.
\begin{definition}
For $1< p<\infty$ and a nonnegative locally integrable function $\omega$ on $\mathbb{R}^n$, $\omega$ is in the
Muckenhoupt $A_{p}$ class if it satisfies the condition
$$\sup_{Q}\bigg(\frac{1}{|Q|}\int_{Q}\omega(x)dx\bigg)\bigg(\frac{1}{|Q|}\int_{Q}\omega(x)^{-\frac{1}{p-1}}dx\bigg)^{p-1}<\infty.$$
And a weight function $\omega$ belongs to the class $A_{1}$ if there exists $C> 0$ such that for every cube Q,
$$\frac{1}{|Q|}\int_{Q}\omega(x)dx\leq C\mathop\mathrm{ess~inf}_{x\in Q}\omega(x).$$
We write $A_{\infty}=\bigcup_{1\leq p<\infty}A_{p}$.
\end{definition}

\begin{definition}
Let $1\leq p<\infty$. Given a a nonnegative locally integrable function $\omega$, the weighted $\mathrm{BMO}$ space $\mathrm{BMO}^{p}(\omega)$ is defined be the set of all functions $f\in L^{1}_{\mathrm{loc}}(\mathbb{R}^{n})$ such that
$$\|f\|_{\mathrm{BMO}^{p}(w)}:=\sup_{Q}\bigg(\frac{1}{w(Q)}\int_{Q}|f(y)-f_{Q}|^{p}\omega(y)^{1-p}dy\bigg)^{1/p}<\infty,$$
where the supremum is taken over all cubes $Q\subset \mathbb{R}^{n}$ and $\omega(Q)=\int_{Q}\omega(x)dx$. We write $\mathrm{BMO}^{1}(\omega)=\mathrm{BMO}(\omega)$ simple.
\end{definition}

{\bf Remark}
For $1\leq p<\infty$ and $\omega\in A_{1}$, Garc\'{i}a-Cuerva \cite{G} proved that $\mathrm{BMO}(\omega)=\mathrm{BMO}^{p}(\omega)$ with equivalence of the corresponding norms.

Standard real analysis tools as the weighted maximal function $M_{\omega}(f)$, the sharp maximal function $M^{\sharp}(f)$ carries over to this context, namely,
$$M_{\omega}(f)(x)=\sup_{Q\ni x}\frac{1}{\omega(Q)}\int_{Q}|f(y)|\omega(y)dy;$$
$$M^{\sharp}(f)(x)=\sup_{Q\ni x}\inf_{c}\frac{1}{|Q|}\int_{Q}|f(y)-c|dy\approx \sup_{Q\ni x}\frac{1}{|Q|}\int_{Q}|f(y)-f_{Q}|dy.$$
A variant of weighted maximal function and sharp maximal operator $M_{\omega,s}(f)(x)=\big(M_{\omega}(f^{s})\big)^{1/s}$ and $M_{\delta}^{\sharp}(f)(x)=\big(M^{\sharp}(f^{\delta})(x)\big)^{1/\delta}$, which will become the main tool in our scheme.

\section{Main lemmas}

To prove Theorem \ref{thm1.1} and Theorem \ref{thm1.2}, we need the following results.

\begin{lemma}\label{lem1}
Let $\omega\in A_{1}$, $\vec{b}=(b,b)$ and $b\in {\rm BMO}(\omega)$. Then
\begin{eqnarray*}
M^{\sharp}_{\frac{1}{3}}\big(\mathcal{M}_{\Pi\vec b}(\vec{f})\big)(x)&\lesssim& \|b\|^2_{{\rm BMO}(\omega)}\omega(x)^2 M(\mathcal{M}(\vec{f})(x))\\
&&+\|b\|^{2}_{{\rm BMO}(\omega)}\omega(x)^{2}\prod_{i=1}^{2}M_{\omega,s}(f_{i})(x)\\
&&+\sum_{i=1}^{2}\|b\|_{{\rm BMO}(\omega)}\omega(x)M_{\frac{1}{2}}(\mathcal{M}^{(i)}_{b}(\vec{f}))(x),
\end{eqnarray*}
for any $1<s<\infty$ and bounded compact supported functions $f_{1},f_{2}$.
\end{lemma}

\begin{proof}
First of all, we give the definition of the following auxiliary maximal function, which has been studied in \cite{ST1} and \cite{ST2} for the linear case.  Let $\varphi(x)\geq 0$ be a smooth function such that $\varphi_{\epsilon}(t)=\epsilon^{-2n}\varphi(\frac{t}{\epsilon})$, $|\varphi'(t)|\lesssim t^{-1}$ and $\chi_{[0,1]}(t)\leq \varphi(t)\leq \chi_{[0,2]}(t)$.

Let $$\Phi(f_{1},f_{2})(x)=\sup_{\epsilon>0}\int_{\mathbb{R}^n}\int_{\mathbb{R}^n}\varphi_{\epsilon}(|x-y_{1}|+|x-y_{2}|)
\prod_{i=1}^{2}|f_{i}(y_{i})|dy_{1}dy_{2},$$
and
$$\Phi_{\Pi\vec{b}}(f_{1},f_{2})(x)=\sup_{\epsilon>0}\int_{\mathbb{R}^n}\int_{\mathbb{R}^n}\varphi_{\epsilon}(|x-y_{1}|+|x-y_{2}|)
\prod_{i=1}^{2}|b(x)-b(y_{i})||f_{i}(y_{i})|dy_{1}dy_{2}.$$

We first show that
$$\Phi(f_{1},f_{2})(x)\approx \mathcal{M}(f_{1},f_{2})(x).$$
In fact, let $B_{\epsilon}=\{y\in \mathbb{R}^n:|x-y|\leq \epsilon\}$. It is easy to see that
$$B_{\frac{\epsilon}{2}}\times B_{\frac{\epsilon}{2}}\subset \big\{(y_{1},y_{2}):|x-y_{1}|+|x-y_{2}|\leq \epsilon\big\}\subset B_{\epsilon}\times B_{\epsilon}.$$
The bounded compact supported condition of $\varphi$ gives
\begin{eqnarray*}
\Phi(f_{1},f_{2})(x)&=&\sup_{\epsilon>0}\int_{\mathbb{R}^n}\int_{\mathbb{R}^n}\varphi_{\epsilon}(|x-y_{1}|+|x-y_{2}|)|f_{1}(y_{1})||f_{2}(y_{2})|dy_{1}dy_{2}\\
&\leq&\sup_{\epsilon>0}\frac{1}{\epsilon^{2n}}\int_{B_{\epsilon}}\int_{B_{\epsilon}}\varphi\Big(\frac{|x-y_{1}|+|x-y_{2}|}{\epsilon}\Big)|f_{1}(y_{1})||f_{2}(y_{2})|dy_{1}dy_{2}\\
&\lesssim& \mathcal{M}(f_{1},f_{2})(x)
\end{eqnarray*}
and
\begin{eqnarray*}
\Phi(f_{1},f_{2})(x)&\geq&\sup_{\epsilon>0}\frac{1}{\epsilon^{2n}}\int_{B_{\frac{\epsilon}{2}}}\int_{B_{\frac{\epsilon}{2}}}\varphi\Big(\frac{|x-y_{1}|+|x-y_{2}|}{\epsilon}\Big)|f_{1}(y_{1})||f_{2}(y_{2})|dy_{1}dy_{2}\\
&\gtrsim& \mathcal{M}(f_{1},f_{2})(x).
\end{eqnarray*}
We can also obtain that $\Phi_{\Pi\vec{b}}(f_{1},f_{2})(x)\approx \mathcal{M}_{\Pi\vec{b}}(f_{1},f_{2})(x).$

Now, we shall estimate the sharp maximal function of the auxiliary maximal function. Let $Q$ be a cube and $x\in Q$. Then, for any $z\in Q$ we have
\begin{eqnarray*}
\big|\Phi_{\Pi\vec{b}}(f_{1},f_{2})(z)-c_{Q}\big|&\leq& \big|b(z)-b_{Q}|^{2}\Phi(f_{1},f_{2})(z)\\
&&+\big||b(z)-b_{Q}| \Phi(f_{1},(b-b_{Q})f_{2})(z)\big|\\
&&+\big||b(z)-b_{Q}|\Phi((b-b_{Q})f_{1},f_{2})(z)\big|\\
&&+\big|\Phi((b-b_{Q})f_{1},(b-b_{Q})f_{2})(z)-c_{Q}\big|\\
&=:&A^{Q}_{1}(z)+A^{Q}_{2}(z)+A^{Q}_{3}(z)+A^{Q}_{4}(z),
\end{eqnarray*}
where $c_{Q}=\big(\Phi((b-b_{Q})f^{\infty}_{1},(b-b_{Q})f^{\infty}_{2})\big)_{Q}$ and $f^{\infty}_{i}$ will be defined later.

Therefore,
\begin{eqnarray*}
\bigg(\frac{1}{|Q|}\int_{Q}\Big|\big|\Phi_{\Pi\vec{b}}(f_{1},f_{2})(z)\big|^{\delta}-|c_{Q}|^{\delta}\Big|dz\bigg)^{1/\delta}
&\lesssim&\bigg(\frac{1}{|Q|}\int_{Q}\Big|\Phi_{\Pi\vec{b}}(f_{1},f_{2})(z)-c_{Q}|^{\delta}\Big|\bigg)^{1/\delta}\\
&\lesssim&\sum_{j=1}^{4}A_{j},
\end{eqnarray*}
where $A_{j}=\big(\frac{1}{|Q|}\int_{Q}\big(A_{j}^{Q}(z)\big)^{\delta}dz\big)^{1/\delta}, j=1,2,3,4$ and taking $\delta=1/3$.

Let us consider first the term $A_{1}$. By averaging $A^{Q}_{1}$ over $Q$, we get
\begin{eqnarray*}
A_{1}&=&\bigg(\frac{1}{|Q|}\int_{Q}\Big(\big|b(z)-b_{Q}|^{2}\Phi(f_{1},f_{2})(z)\Big)^{1/3}dz\bigg)^{3}\\
&\lesssim&\|b\|^{2}_{{\rm BMO}(\omega)}\frac{\omega(Q)^{2}}{|Q|^{2}}
\cdot\frac{1}{|Q|}\int_{Q}\mathcal{M}(f_{1},f_{2})(z)dz\\
&\lesssim&\|b\|^{2}_{{\rm BMO}(\omega)}\omega(x)^{2}M(\mathcal{M}(f_{1},f_{2}))(x).
\end{eqnarray*}

Let us consider next the term $A_{2}$. We write
\begin{eqnarray*}
A^{Q}_{2}(z)&=&\big|b(z)-b_{Q}|\Phi(f_{1},|b-b_{Q}|f_{2})(z)\\
&\leq&\big|b(z)-b_{Q}|\Phi(f_{1},\big(|b(z)-b_{Q}|+|b(z)-b|\big)f_{2})(z)\\
&\leq&\big|b(z)-b_{Q}|^{2}\mathcal{M}(f_{1},f_{2})(z)+\big|b(z)-b_{Q}|\mathcal{M}^{(2)}_{b}(f_{1},f_{2})(z)\\
&=:&A^{Q}_{21}(z)+A^{Q}_{22}(z).
\end{eqnarray*}

For $A^{Q}_{21}(z)$, the fact that $\Phi(f_{1},f_{2})(z)\lesssim \mathcal{M}(f_{1},f_{2})(z)$ gives
\begin{eqnarray*}
A_{21}&:=&\bigg(\frac{1}{|Q|}\int_{Q}\big(A^{Q}_{21}(z)\big)^{\delta}dz\bigg)^{1/\delta}\\
&\lesssim&\Big(\frac{\omega(Q)}{|Q|}\|b\|_{{\rm BMO}(\omega)}\Big)^{2}\frac{1}{|Q|}\int_{Q}\mathcal{M}(f_{1},f_{2})(z)dz\\
&\lesssim&\|b\|^{2}_{{\rm BMO}(\omega)}\omega(x)^{2}M(\mathcal{M}(f_{1},f_{2}))(x).
\end{eqnarray*}

For $A^{Q}_{22}(z)$,
\begin{eqnarray*}
A_{22}&:=&\bigg(\frac{1}{|Q|}\int_{Q}\big(A^{Q}_{22}(z)\big)^{\delta}dz\bigg)^{1/\delta}\\
&\lesssim&\omega(x)\|b\|_{{\rm BMO}(\omega)}\frac{1}{|Q|^{2}}\bigg(\int_{Q}\big|\mathcal{M}^{(2)}_{b}(f_{1},f_{2})(z)\big|^{1/2}\bigg)^{2}\\
&\lesssim&\omega(x)\|b\|_{{\rm BMO}(\omega)}M_{1/2}(\mathcal{M}^{(2)}_{b}(f_{1},f_{2}))(x).
\end{eqnarray*}

The same process also follows that
$$A_{3}\lesssim\|b\|^{2}_{{\rm BMO}(\omega)}\omega(x)^{2}M(\mathcal{M}(f_{1},f_{2}))(x)+ \omega(x)\|b\|_{{\rm BMO}(\omega)}M_{1/2}(\mathcal{M}^{(1)}_{b}(f_{1},f_{2}))(x).$$

To estimate $A_{4}$, we split $f_{j}$ to $f_{j}=f_{j}^{0}+f_{j}^{\infty}$ with $f_{j}^{0}=f_{j}\chi_{2Q}$. We write
\begin{eqnarray*}
A^{Q}_{4}&\leq& \big|\Phi((b-b_{Q})f^{0}_{1},(b-b_{Q})f^{0}_{2})(z)\big|\\
&&+\big|\Phi(((b-b_{Q})f^{0}_{1},(b-b_{Q})f^{\infty}_{2})(z)\big|\\
&&+\big|\Phi((b-b_{Q})f^{\infty}_{1},(b-b_{Q})f^{0}_{2})(z)\big|\\
&&+\big|\Phi((b-b_{Q})f^{\infty}_{1},(b-b_{Q})f^{\infty}_{2})(z)-c_{Q}\big|\\
&=:&A^{Q}_{41}(z)+A^{Q}_{42}(z)+A^{Q}_{43}(z)+A^{Q}_{44}(z).
\end{eqnarray*}
Then
\begin{eqnarray*}
A_{4}&\leq& \bigg(\frac{1}{|Q|}\int_{Q}\big(\sum_{j=1}^{4}A^{Q}_{4j}(z)\big)^{\delta}dz\bigg)^{1/\delta}\\
&\lesssim& \sum_{j=1}^{4}\bigg(\frac{1}{|Q|}\int_{Q}\big(A^{Q}_{4j}(z)\big)^{\delta}dz\bigg)^{1/\delta}\\
&\lesssim&\sum_{j=1}^{4}A_{4j}.
\end{eqnarray*}

By Kolmogorov inequality and the fact that $\mathcal{M}$ is bounded from $L^{1}\times L^{1}$ to $L^{1/2,\infty}$, we have
\begin{eqnarray*}
A_{41}&\leq& \frac{C}{|Q|^{2}}\|\Phi((b-b_{Q})f^{0}_{1},(b-b_{Q})f^{0}_{2})\|_{L^{1/2,\infty}}\\
&\leq&\frac{C}{|Q|^{2}}\|\mathcal{M}((b-b_{Q})f^{0}_{1},(b-b_{Q})f^{0}_{2})\|_{L^{1/2,\infty}}\\
&\leq&\frac{C}{|Q|^{2}}\prod_{i=1}^{2}\int_{2Q}|b(y_{i})-b_{Q}||f_{i}(y_{i})|dy_{i}\\
&\leq&\frac{C}{|Q|^{2}}\prod_{i=1}^{2}\bigg(\int_{2Q}|b(y_{i})-b_{Q}|^{s'}\omega^{1-s'}(y_{i})dy_{i}\bigg)^{1/s'}
\bigg(\int_{2Q}|f_{i}(y_{i})|^{s}\omega(y_{i})dy_{i}\bigg)^{1/s}\\
&\lesssim&\prod_{i=1}^{2}\|b\|_{{\rm BMO}^{s'}(\omega)}\omega(x)M_{\omega,s}(f_{i})(x).
\end{eqnarray*}

For $A_{42}$, it is easy to see that
$$\varphi_{\epsilon}(|z-y_{1}|+|z-y_{2}|)\lesssim \frac{1}{\big(|z-y_{1}|+|z-y_{2}|\big)^{2n}},$$
then
\begin{eqnarray*}
A_{42}&\lesssim& \frac{1}{|Q|}\int_{Q}\int_{2Q}\int_{\mathbb{R}^n\backslash 2Q}
\frac{|b(y_{1})-b_{Q}||f_{1}(y_{1})||b(y_{2})-b_{Q}||f_{2}(y_{2})|}{\big(|z-y_{1}|+|z-y_{2}|\big)^{2n}}dy_{1}dy_{2}dz\\
&\lesssim& \int_{Q}\int_{2Q}\int_{\mathbb{R}^n\backslash 2Q}
\frac{|b(y_{1})-b_{Q}||f_{1}(y_{1})||b(y_{2})-b_{Q}||f_{2}(y_{2})|}{\big(|z-y_{1}|+|z-y_{2}|\big)^{2n}}dy_{1}dy_{2}dz\\
&\lesssim& \frac{1}{|Q|}\int_{2Q}|b(y_{1})-b_{Q}||f_{1}(y_{1})|dy_{1}\int_{Q}\int_{\mathbb{R}^n\backslash 2Q}
\frac{|b(y_{2})-b_{Q}||f_{2}(y_{2})|}{|z-y_{2}|^{2n}}dy_{2}dz\\
&\lesssim&\|b\|_{{\rm BMO}^{s'}(\omega)}\omega(x)M_{\omega,s}(f_{1})(x)
\sum_{k=1}^{\infty}\frac{2^{-kn}}{|2^{k}Q|}\int_{2^{k}Q}|b(y_{2})-b_{Q}||f_{2}(y_{2})|dy_{2}\\
&\lesssim&\|b\|_{{\rm BMO}^{s'}(\omega)}\omega(x)M_{\omega,s}(f_{1})(x)
\sum_{k=1}^{\infty}\frac{2^{-kn}}{|2^{k}Q|}\\
&&\times\bigg[\int_{2^{k}Q}|b(y_{1})-m_{2^{k}Q}(b)||f_{2}(y_{2})|dy_{2}+\int_{2^{k}Q}|m_{2^{k}Q}(b)-b_{Q}||f_{2}(y_{2})|dy_{2}\bigg]\\
&\lesssim&\|b\|_{{\rm BMO}^{s'}(\omega)}\omega(x)M_{\omega,s}(f_{1})(x)
\sum_{k=1}^{\infty}\frac{2^{-kn}}{|2^{k}Q|}\\
&&\times\bigg[\|b\|_{{\rm BMO}^{s'}(\omega)}\omega(x)M_{\omega,s}(f_{2})(x)+k\|b\|_{{\rm BMO}(\omega)}\omega(x)M(f_{2})(x)\bigg]\\
&\lesssim&\prod_{i=1}^{2}\|b\|_{{\rm BMO}^{s'}(\omega)}\omega(x)M_{\omega,s}(f_{i})(x).
\end{eqnarray*}

Similarly, for $A_{43}$, we have
$$A_{43}\lesssim\prod_{i=1}^{2}\|b\|_{{\rm BMO}^{s'}(\omega)}\omega(x)M_{\omega,s}(f_{i})(x).$$

For $|z-z'|\leq \frac{1}{2}\max\{|z-y_{1}|,|z-y_{2}|\}$,
$$\big|\varphi_{\epsilon}(|z-y_{1}|+|z-y_{2}|)-\varphi_{\epsilon}(|z'-y_{1}|+|z'-y_{2}|)\big|\lesssim \frac{|z-z'|}{\big(|z-y_{1}|+|z-y_{2}|\big)^{2n+1}}.$$
Therefore,
\begin{eqnarray*}
&&\big|\Phi((b(z)-b)f^{\infty}_{1},(b(z)-b)f^{\infty}_{2})(z)
-\Phi((b(z)-b)f^{\infty}_{1},(b(z)-b)f^{\infty}_{2})(z')\big|\\
&&\lesssim \sup_{\epsilon>0}\int_{\mathbb{R}^n\backslash 2Q}\int_{\mathbb{R}^n\backslash 2Q}
\Big|\varphi_{\epsilon}(|z-y_{1}|+|z-y_{2}|)-\varphi_{\epsilon}(|z'-y_{1}|+|z'-y_{2}|)\Big|\\
&&\quad \times\prod_{i=1}^{2}|b(y_{i})-b_{Q}||f_{i}(y_{i})|dy_{1}dy_{2}\\
&&\lesssim\prod_{i=1}^{2}\int_{\mathbb{R}^n\backslash 2Q}
\frac{|z-z'|^{\epsilon_{i}}}{|z-y_{i}|^{n+\epsilon_{i}}}|b(y_{i})-b_{Q}||f_{i}(y_{i})|dy_{i}\\
&&\lesssim\prod_{i=1}^{2}\sum_{k=1}^{\infty}\frac{-2^{kn\epsilon_{i}}}{|2^{k}Q|}\int_{2^{k}Q}|b(y_{i})-b_{Q}||f_{i}(y_{i})|dy_{i}\\
&&\lesssim\prod_{i=1}^{2}\|b\|_{{\rm BMO}^{s'}(\omega)}\omega(x)M_{\omega,s}(f_{i})(x),
\end{eqnarray*}
where $\epsilon_{1},\epsilon_{2}>0$ with $\epsilon_{1}+\epsilon_{2}=1$.

Collecting our estimates, we have shown that
\begin{eqnarray*}
M^{\sharp}_{\frac{1}{3}}\big(\mathcal{M}_{\Pi\vec b}(\vec{f})\big)(x)&\lesssim& \|b\|^2_{{\rm BMO}(\omega)}\omega(x)^2 M(\mathcal{M}(\vec{f})(x))\\
&&+\|b\|^{2}_{{\rm BMO}(\omega)}\omega(x)^{2}\prod_{i=1}^{2}M_{\omega,s}(f_{i})(x)\\
&&+\sum_{i=1}^{2}\|b\|_{{\rm BMO}(\omega)}\omega(x)M_{\frac{1}{2}}(\mathcal{M}^{(i)}_{b}(\vec{f}))(x),
\end{eqnarray*}
for any $1<s<\infty$ and bounded compact supported functions $f_{1},f_{2}$.
\end{proof}

\begin{lemma}\label{lem2}
Let $\omega\in A_{1}$, $\vec{b}=(b,b)$ and $b\in {\rm BMO}(\omega)$. Then there exist a constant $C$ such that
\begin{eqnarray*}
M^{\sharp}_{\frac{1}{2}}\big(\mathcal{M}^{(1)}_{b}(\vec{f})\big)(x)&\lesssim& \|b\|_{{\rm BMO}(\omega)}\omega(x)M(\mathcal{M}(\vec{f})(x))\\
&&+\|b\|_{{\rm BMO}(\omega)}\omega(x)M_{\omega,s}(f_{1})(x)M(f_{2})(x),
\end{eqnarray*}
for any $1<s<\infty$ and bounded compact supported functions $f_{1},f_{2}$.
\end{lemma}

\begin{proof}
Let $Q$ be a cube and $x\in Q$. Then, for $z\in Q$ we have
\begin{eqnarray*}
\big|\Phi^{(1)}_{b}(f_{1},f_{2})(z)-c_{Q}\big|&\leq &\big|b(z)-b_{Q}|\Phi(f_{1},f_{2})(z)\\
&&+\big|\Phi((b-b_{Q})f_{1},f_{2})(z)-c_{Q}\big|\\
&=:&B^{Q}_{1}(z)+B^{Q}_{2}(z).
\end{eqnarray*}

Therefore,
\begin{eqnarray*}
&&\bigg(\frac{1}{|Q|}\int_{Q}\Big|\big|\Phi^{(1)}_{b}(f_{1},f_{2})(z)\big|^{1/2}-|c_{Q}|^{1/2}\Big|dz\bigg)^{2}\\
&&\lesssim\bigg(\frac{1}{|Q|}\int_{Q}\big|\Phi_{\Pi\vec{b}}(f_{1},f_{2})(z)-c_{Q}\big|^{1/2}dz\bigg)^{2}\\
&&\lesssim\sum_{j=1}^{2}B_{j},
\end{eqnarray*}
where $B_{j}=\big(\frac{1}{|Q|}\int_{Q}\big(B_{j}^{Q}(z)\big)^{\delta}dz\big)^{1/\delta}, j=1,2$.

Let us consider first the term $B_{1}$. By averaging $B^{Q}_{1}$ over $Q$, we get
\begin{eqnarray*}
B_{1}&=&\bigg(\frac{1}{|Q|}\int_{Q}\Big(\big|b(z)-b_{Q}|\Phi(f_{1},f_{2})(z)\Big)^{1/2}dz\bigg)^{2}\\
&\lesssim&\|b\|_{{\rm BMO}(\omega)}\frac{\omega(Q)}{|Q|}
\cdot\frac{1}{|Q|}\int_{Q}\mathcal{M}(f_{1},f_{2})(z)dz\\
&\lesssim&\|b\|_{{\rm BMO}(\omega)}\omega(x)M(\mathcal{M}(f_{1},f_{2}))(x).
\end{eqnarray*}

Let us consider next the term $B_{2}$. We split $f_{j}$ to $f_{j}=f_{j}^{0}+f_{j}^{\infty}$ with $f_{j}^{0}=f_{j}\chi_{2Q}$. We write
\begin{eqnarray*}
B^{Q}_{2}&\leq& \big|\Phi((b-b_{Q})f^{0}_{1},f^{0}_{2})(z)\big|+\big|\Phi(((b-b_{Q})f^{0}_{1},f^{\infty}_{2})(z)\big|\\
&&+\big|\Phi((b-b_{Q})f^{\infty}_{1},f^{0}_{2})(z)\big|+\big|\Phi((b-b_{Q})f^{\infty}_{1},f^{\infty}_{2})(z)-c_{Q}\big|\\
&=:&B^{Q}_{21}(z)+B^{Q}_{22}(z)+B^{Q}_{23}(z)+B^{Q}_{24}(z).
\end{eqnarray*}

By Kolmogorov inequality and the fact that $\mathcal{M}$ is bounded from $L^{1}\times L^{1}$ to $L^{1/2,\infty}$, we have
\begin{eqnarray*}
B_{21}&\leq& \frac{C}{|Q|^{2}}\|\Phi((b-b_{Q})f^{0}_{1},f^{0}_{2})\|_{L^{1/2,\infty}}\\
&\lesssim&\frac{1}{|Q|^{2}}\|\mathcal{M}((b-b_{Q})f^{0}_{1},f^{0}_{2})\|_{L^{1/2,\infty}}\\
&\lesssim&\frac{1}{|Q|^{2}}\int_{2Q}|b(y_{1})-b_{Q}||f_{1}(y_{1})|dy_{1}\int_{2Q}|f_{2}(y_{2})|dy_{2}\\
&\lesssim&\|b\|_{{\rm BMO}^{s'}(\omega)}\omega(x)M_{\omega,s}(f_{1})(x)M(f_{2})(x).
\end{eqnarray*}

For $B_{22}$,
\begin{eqnarray*}
B_{22}&\lesssim& \frac{1}{|Q|}\int_{Q}\int_{2Q}\int_{\mathbb{R}^n\backslash 2Q}
\frac{|b(y_{1})-b_{Q}||f_{1}(y_{1})||f_{2}(y_{2})|}{\big(|z-y_{1}|+|z-y_{2}|\big)^{2n}}dy_{1}dy_{2}dz\\
&\lesssim&\frac{1}{|Q|} \int_{Q}\int_{2Q}\int_{\mathbb{R}^n\backslash 2Q}
\frac{|b(y_{1})-b_{Q}||f_{1}(y_{1})||f_{2}(y_{2})|}{\big(|z-y_{1}|+|z-y_{2}|\big)^{2n}}dy_{1}dy_{2}dz\\
&\lesssim& \frac{1}{|Q|}\int_{2Q}|b(y_{1})-b_{Q}||f_{1}(y_{1})|dy_{1}\int_{Q}\int_{\mathbb{R}^n\backslash 2Q}
\frac{|f_{2}(y_{2})|}{|z-y_{2}|^{2n}}dy_{2}dz\\
&\lesssim&\|b\|_{{\rm BMO}^{s'}(\omega)}\omega(x)M_{\omega,s}(f_{1})(x)
\sum_{k=1}^{\infty}\frac{2^{-kn}}{|2^{k}Q|}\int_{2^{k}Q}|f_{2}(y_{2})|dy_{2}\\
&\lesssim&\|b\|_{{\rm BMO}^{s'}(\omega)}\omega(x)M_{\omega,s}(f_{1})(x)M(f_{2})(x).
\end{eqnarray*}

For $B_{23}$, we have
\begin{eqnarray*}
B_{23}&\lesssim& \frac{1}{|Q|}\int_{Q}\int_{\mathbb{R}^n\backslash 2Q}\int_{2Q}
\frac{|b(y_{1})-b_{Q}||f_{1}(y_{1})||f_{2}(y_{2})|}{\big(|z-y_{1}|+|z-y_{2}|\big)^{2n}}dy_{1}dy_{2}dz\\
&\lesssim&\frac{1}{|Q|} \int_{Q}\int_{2Q}\int_{\mathbb{R}^n\backslash 2Q}
\frac{|b(y_{1})-b_{Q}||f_{1}(y_{1})||f_{2}(y_{2})|}{\big(|z-y_{1}|+|z-y_{2}|\big)^{2n}}dy_{1}dy_{2}dz\\
&\lesssim& \frac{1}{|Q|}\int_{Q}\int_{\mathbb{R}^n\backslash 2Q}\frac{|b(y_{1})-b_{Q}||f_{1}(y_{1})|}{|z-y_{1}|^{2n}}dy_{1}dz\int_{2Q}
|f_{2}(y_{2})|dy_{2}\\
&\lesssim&\|b\|_{{\rm BMO}^{s'}(\omega)}\omega(x)M_{\omega,s}(f_{1})(x)M(f_{2})(x).
\end{eqnarray*}

Concerning the last estimate for $B_{24}$. For any $z'\in Q$ and $y_{1},y_{2}\in \mathbb{R}^n\backslash 2Q$, we have
\begin{eqnarray*}
&&\big|\Phi((b-b_{Q})f^{\infty}_{1},f^{\infty}_{2})(z)
-\Phi((b-b_{Q})f^{\infty}_{1},f^{\infty}_{2})(z')\big|\\
&&\lesssim \sup_{\epsilon>0}\int_{\mathbb{R}^n\backslash 2Q}\int_{\mathbb{R}^n\backslash 2Q}
\Big|\varphi_{\epsilon}(|z-y_{1}|+|z-y_{2}|)-\varphi_{\epsilon}(|z'-y_{1}|+|z'-y_{2}|)\Big|\\
&&\quad \times|b(y_{1})-b_{Q}||f_{1}(y_{1})||f_{2}(y_{2})|dy_{1}dy_{2}\\
&&\lesssim \int_{\mathbb{R}^n\backslash 2Q}\frac{|b(y_{1})-b_{Q}||f_{1}(y_{1})|}{|z-y_{1}|^{2n+\epsilon_{1}}}dy_{1}\int_{\mathbb{R}^n\backslash 2Q}
\frac{|f_{2}(y_{2})|}{|z-y_{2}|^{\epsilon_{2}}}dy_{2}\\
&&\lesssim\sum_{k=2}^{\infty}\frac{2^{-kn\epsilon_{1}}}{|2^{k}Q|}\int_{2^{k}Q}|b(y_{1})-b_{Q}||f_{1}(y_{1})|dy_{1}
\sum_{i=2}^{\infty}\frac{2^{-kn\epsilon_{2}}}{|2^{k}Q|}\int_{2^{i}Q}|f_{2}(y_{2})|dy_{2}\\
&&\lesssim\|b\|_{{\rm BMO}^{s'}(\omega)}\omega(x)M_{\omega,s}(f_{1})(x)M(f_{2})(x).
\end{eqnarray*}
where $\epsilon_{1},\epsilon_{2}>0$ with $\epsilon_{1}+\epsilon_{2}=1$. Taking the mean over $Q$ for $z$ and $z'$ respectively, we obtain
\begin{eqnarray*}
B_{24}&\lesssim&\frac{1}{|Q|}\int_{Q}\big|\Phi((b-b_{Q})f^{\infty}_{1},f^{\infty}_{2})(z)
-c_{Q}\big|dz\\
&\lesssim&\frac{1}{|Q|}\int_{Q}\frac{1}{|Q|}\int_{Q}\big|\Phi((b-b_{Q})f^{\infty}_{1},f^{\infty}_{2})(z)
-\Phi((b-b_{Q})f^{\infty}_{1},f^{\infty}_{2})(z')\big|dzdz'\\
&\lesssim& \|b\|_{{\rm BMO}^{s'}(\omega)}\omega(x)M_{\omega,s}(f_{1})(x)M(f_{2})(x).
\end{eqnarray*}

Collecting our estimates, we have shown that
\begin{eqnarray*}
M^{\sharp}_{\frac{1}{2}}\big(\mathcal{M}^{(1)}_{b}(\vec{f})\big)(x)&\lesssim& \|b\|_{{\rm BMO}(\omega)}\omega(x)M(\mathcal{M}(\vec{f})(x))\\
&&+\|b\|_{{\rm BMO}(\omega)}\omega(x)M_{\omega,s}(f_{1})(x)M(f_{2})(x),
\end{eqnarray*}
for any $1<s<\infty$ and bounded compact supported functions $f_{1},f_{2}$.
\end{proof}

Similarly, we have
\begin{lemma}\label{lem3}
Let $\omega\in A_{1}$, $\vec{b}=(b,b)$ and $b\in {\rm BMO}(\omega)$. Then there exist a constant $C$ such that
\begin{eqnarray*}
M^{\sharp}_{\frac{1}{2}}\big(\mathcal{M}^{(2)}_{b}(\vec{f})\big)(x)&\lesssim& \|b\|_{{\rm BMO}(\omega)}\omega(x)M(\mathcal{M}(\vec{f})(x))\\
&&+\|b\|_{{\rm BMO}(\omega)}\omega(x)M_{\omega,s}(f_{2})(x)M(f_{1})(x),
\end{eqnarray*}
for any $1<s<\infty$ and bounded compact supported functions $f_{1},f_{2}$.
\end{lemma}

\begin{lemma}\label{lem4}
Let $\omega\in A_{1}$ and $0<p<\infty$. Then $\omega^{1-p}\in A_{\infty}$.
\end{lemma}
\begin{proof}
If $0<p\leq 1$, then $1-p\in [0,1)$. It is easy to see that $\omega^{1-p}\in A_{1}\subset A_{\infty}$.

If $1<p<\infty$, if follows from $\omega\in A_{1}\subset A_{p}$ that $\omega^{1-p}\in A_{p'}\subset A_{\infty}$.
\end{proof}

\begin{lemma}\label{lem5}
Let $\omega\in A_{1}$, $1<s<p_{1},p_{2}<\infty$ and $1/p=1/p_{1}+1/p_{2}$. Then both $\mathcal{M}(\vec{f})$ and $\prod_{i=1}^{2}M_{\omega,s}(f_{i})$ are bounded from $L^{p_{1}}(\omega)\times L^{p_{2}}(\omega)$ to $L^{p}(\omega)$.
\end{lemma}
\begin{proof}
From the fact that $M(f)(x)\lesssim M_{\omega,s}(f)(x)$ and $M_{\omega,s}(f)(x)$ is bounded on $L^{p}(\omega)$ for $1<s<p_{1},p_{2}<\infty$, it is easy to obtain that both $\mathcal{M}(\vec{f})$ and $\prod_{i=1}^{2}M_{\omega,s}(f_{i})$ are bounded from $L^{p_{1}}(\omega)\times L^{p_{2}}(\omega)$ to $L^{p}(\omega)$.
\end{proof}

The following relationships between $M_{\delta}$ and $M^{\sharp}$ to be used is a version of the classical ones
due to Fefferman and Stein \cite{FS}.
\begin{lemma}\label{lem6}
Let $0<p,\delta<\infty$ and $\omega\in A_{\infty}$. There exist a positive $C$ such that
$$\int_{\mathbb{R}^n}(M_{\delta}f(x))^{p}\omega(x)dx\leq C\int_{\mathbb{R}^n}(M^{\sharp}_{\delta}f(x))^{p}\omega(x)dx,$$
for any smooth function $f$ for which the left-hand side is finite.
\end{lemma}

\begin{lemma}\label{lem7}
Let $Q_{0}$ be any fixed cube and $b$ be a locally integral function. Then, for any $x\in Q_{0}$, we get
\begin{equation}\label{l7.1}
\mathcal{M}(\chi_{Q_{0}},\chi_{Q_{0}})(x)\equiv 1;
\end{equation}
\begin{equation}\label{l7.2}
\mathcal{M}(b\chi_{Q_{0}},\chi_{Q_{0}})(x)=\mathcal{M}(\chi_{Q_{0}},b\chi_{Q_{0}})(x)=\mathcal{M}_{Q_{0}}(b)(x);
\end{equation}
\begin{equation}\label{l7.3}
\mathcal{M}(b\chi_{Q_{0}},b\chi_{Q_{0}})(x)=\mathcal{M}^{2}_{Q_{0}}(b)(x),
\end{equation}
where $M_{Q_{0}}(b)(x)=\sup_{Q_{0}\supset Q\ni x}\frac{1}{|Q|}\int_{Q}|b(y)|dy.$
\end{lemma}
\begin{proof}
We only give the proof of (\ref{l7.3}) and the proof of (\ref{l7.1}),(\ref{l7.2}) are similar. For any $x\in Q_{0}$, we have
\begin{eqnarray*}
M^{2}_{Q_{0}}(b)(x)&=&\Big(\sup_{Q_{0}\supset Q\ni x}\frac{1}{|Q|}\int_{Q}|b(y)|dy\Big)^{2}\\
&=&\sup_{Q_{0}\supset Q\ni x}\frac{1}{|Q|}\int_{Q}|b(y_{1})|\chi_{Q_{0}}(y_{1})dy_{1}\cdot\frac{1}{|Q|}\int_{Q}|b(y_{2})|\chi_{Q_{0}}(y_{2})dy_{2}\\
&\leq& \mathcal{M}(b\chi_{Q_{0}},b\chi_{Q_{0}})(x).
\end{eqnarray*}

On the other hand, for any cube $Q\subset \mathbb{R}^n$, we can construct a cube $Q_{1}$ such that
$$Q_{0}\supset Q_{1}\supset Q_{0}\cap Q\ni x$$
and $|Q_{1}|\leq |Q|$. Therefore,
$$\frac{1}{|Q|}\int_{Q\cap Q_{0}}|b(y)|dy\leq\frac{1}{|Q_{1}|}\int_{Q_{1}}|b(y)|dy\leq M_{Q_{0}}(b)(x).$$
Thus,
$$\mathcal{M}(b\chi_{Q_{0}},b\chi_{Q_{0}})(x)=\sup_{Q\ni x}\Big(\frac{1}{|Q|}\int_{Q}|b(y)|\chi_{Q_{0}}(y)dy\Big)^{2}\leq  M^{2}_{Q_{0}}(b)(x),$$
then (\ref{l7.3}) is proved.
\end{proof}

\section{Proofs of Theorem \ref{thm1.1} and Theorem \ref{thm1.2}}

{\it Proof of Theorem \ref{thm1.1}.}
$(A1)\Rightarrow (A2)$: It is enough to prove Theorem \ref{thm1.1} for $f_{1},f_{2}$ being bounded functions with compact
support. We observe that to use the Fefferman-Stein inequality,
one needs to verify that certain terms in the left-hand side of the inequalities are finite. Applying a similar argument as in \cite[pp.32-33]{LOPTT}, the boundedness properties of $\mathcal{M}$ and Fatou's lemma, one gets the desired result.

Since Lemma \ref{lem4} and $\omega\in A_{1}$, then $\omega^{1-p}\in A_{\infty}$. By Lemma \ref{lem2} and Lemma \ref{lem3} with $1<s<\min\{p_{1},p_{2}\}$, from a standard argument that we can obtain
\begin{eqnarray*}
\|\mathcal{M}_{\Sigma\vec{b}}(\vec{f})\|_{L^{p}(\omega^{1-p})}&\lesssim& \|M_{\frac{1}{2}}\big(\mathcal{M}_{\Sigma\vec{b}}(\vec{f})\big)\|_{L^{p}(\omega^{1-p})}
\lesssim\|M^{\sharp}_{\frac{1}{2}}\big(\mathcal{M}_{\Sigma\vec{b}}(\vec{f})\big)\|_{L^{p}(\omega^{1-p})}\\
&\lesssim&\|b\|_{{\rm BMO}(\omega)}\bigg(\big\|M\big(\mathcal{M}(\vec{f})\big)\big\|_{L^{p}(\omega)}+\big\|\prod_{i=1}^{2}M_{\omega,s}(f_{i})\big\|_{L^{p}(\omega)}\bigg)\\
&\lesssim& \|b\|_{{\rm BMO}(\omega)}\prod_{i=1}^{2}\|f_{i}\|_{L^{p_{i}}(\omega)}.
\end{eqnarray*}

$(A2)\Rightarrow (A1)$: Let $Q$ be any fixed cube. Suppose that $\mathcal{M}_{\Sigma\vec{b}}$ is bounded from $L^{p_{1}}(\omega)\times L^{p_{2}}(\omega)$ into $L^{p}(\omega^{1-p})$, then
$$\|\mathcal{M}_{\Sigma\vec{b}}(\chi_{Q},\chi_{Q})\|_{L^{p}(\omega^{1-p})}\lesssim \|\chi_{Q}\|_{L^{p_{1}}(\omega)}\|\chi_{Q}\|_{L^{p_{2}}(\omega)}\lesssim \omega(Q)^{\frac{1}{p}},$$
which implies that
\begin{eqnarray*}
&&\frac{2}{\omega(Q)}\int_{Q}|b(x)-b_{Q}|dx\\
&&\leq\frac{1}{\omega(Q)}\int_{Q}|Q|^{-2}\int_{Q}\int_{Q}|b(x)-b(y_{1})|\chi_{Q}(y_{1})\chi_{Q}(y_{2})dy_{1}dy_{2}dx\\
&&\qquad+\frac{1}{\omega(Q)}\int_{Q}|Q|^{-2}\int_{Q}\int_{Q}|b(x)-b(y_{2})|\chi_{Q}(y_{1})\chi_{Q}(y_{2})dy_{1}dy_{2}dx\\
&&\lesssim\frac{1}{\omega(Q)}\int_{Q}\mathcal{M}_{\Sigma\vec{b}}(\chi_{Q},\chi_{Q})(x)dx\\
&&\lesssim\frac{1}{\omega(Q)}\Big(\int_{Q}\big|\mathcal{M}_{\vec b}(\chi_{Q},\chi_{Q})(x)\big|^{p}\omega(x)^{1-p}dx\Big)^{1/p}\Big(\int_{Q}\omega(x)dx\Big)^{1/p'}\\
&&\lesssim \frac{1}{\omega(Q)^{1/p}}\|\mathcal{M}_{\Sigma\vec{b}}(\chi_{Q},\chi_{Q})\|_{L^{p}(\omega^{1-p})}\\
&&\lesssim \|\mathcal{M}_{\Sigma\vec{b}}\|_{L^{p_{1}}(\omega)\times L^{p_{2}}(\omega)\rightarrow L^{p}(\omega^{1-p})}.
\end{eqnarray*}
Thus showing that $b\in {\rm BMO}(\omega)$.

$(A1)\Rightarrow (A3)$: Since $\omega\in A_{1}$,  Lemma \ref{lem4} implies that $\omega^{1-2p}\in A_{\infty}$. From Lemma \ref{lem1}, Lemma \ref{lem2} and Lemma \ref{lem3} with $1<s<\min\{p_{1},p_{2}\}$, we get
\begin{eqnarray*}
\|\mathcal{M}_{\Pi\vec{b}}(\vec{f})\|_{L^{p}(\omega^{1-2p})}&\lesssim& \|M_{\frac{1}{3}}\big(\mathcal{M}_{\Pi\vec{b}}(\vec{f})\big)\|_{L^{p}(\omega^{1-2p})}
\lesssim\|M^{\sharp}_{\frac{1}{3}}\big(\mathcal{M}_{\Pi\vec{b}}(\vec{f})\big)\|_{L^{p}(\omega^{1-2p})}\\
&\lesssim&\|b\|^{2}_{{\rm BMO}(\omega)}\bigg(\big\|M\big(\mathcal{M}(\vec{f})\big)\big\|_{L^{p}(\omega)}+\big\|\prod_{i=1}^{2}M_{\omega,s}(f_{i})\big\|_{L^{p}(\omega)}\bigg)\\
&&+\sum_{i=1}^{2}\|b\|_{{\rm BMO}(\omega)}\big\|M_{\frac{1}{2}}(\mathcal{M}^{(i)}_{b}(\vec{f}))(x)\big\|_{L^{p}(\omega^{1-p})}\\
&\lesssim& \|b\|^{2}_{{\rm BMO}(\omega)}\prod_{i=1}^{2}\|f_{i}\|_{L^{p_{i}}(\omega)}.
\end{eqnarray*}

$(A3)\Rightarrow (A1)$: By $\mathcal{M}_{\Pi\vec{b}}$ is bounded from $L^{p_{1}}(\omega)\times L^{p_{2}}(\omega)$ into $L^{p}(\omega^{1-2p})$, we get
\begin{eqnarray*}
&&\frac{1}{\omega(Q)}\int_{Q}|b(x)-b_{Q}|^{2}\omega(x)^{-1}dx\\
&&\lesssim\frac{1}{\omega(Q)}\int_{Q}\omega(x)^{-1}|Q|^{-2}\int_{Q}\int_{Q}|b(x)-b(y_{1})||b(x)-b(y_{2})|dy_{1}dy_{2}dx\\
&&\lesssim\frac{1}{\omega(Q)}\int_{Q}\mathcal{M}_{\Pi\vec{b}}(\chi_{Q},\chi_{Q})(x)\omega(x)^{-1}dx\\
&&\lesssim\frac{1}{\omega(Q)}\Big(\int_{Q}\big|\mathcal{M}_{\Pi\vec{b}}(\chi_{Q},\chi_{Q})(x)\big|^{p}
\omega(x)^{1-2p}dx\Big)^{1/p}\Big(\int_{Q}\omega(x)dx\Big)^{1/p'}\\
&&\lesssim \frac{1}{\omega(Q)^{1/p}}\|\mathcal{M}_{\Pi\vec{b}}(\chi_{Q},\chi_{Q})\|_{L^{p}(\omega^{1-2p})}\\
&&\lesssim \|\mathcal{M}_{\Pi\vec{b}}\|_{L^{p_{1}}(\omega)\times L^{p_{2}}(\omega)\rightarrow L^{p}(\omega^{1-2p})}.
\end{eqnarray*}

Thus we complete the proof of Theorem \ref{thm1.1}. \qed

\vskip 0.5cm
\noindent

{\it Proof of Theorem \ref{thm1.2}.}
$(B1)\Rightarrow B2)$: By the definition of $\mathcal{M}(\vec{f})$, we have
$$M(bf_{1},f_{2})(x)=M(|b|f_{1},f_{2})(x),\ \  M(f_{1},bf_{2})(x)=M(f_{1},|b|f_{2})(x).$$
Then
\begin{eqnarray*}
&&\big|[b,\mathcal{M}]^{(1)}(\vec{f})(x)-[|b|,\mathcal{M}]^{(1)}(\vec{f})(x)\big|\\
&&\lesssim \Big|b(x)\mathcal{M}(\vec{f})(x)-\mathcal{M}(bf_{1},f_{2})(x)-|b(x)|\mathcal{M}(\vec{f})(x)+\mathcal{M}(|b|f_{1},f_{2})(x)\Big|\\
&&\lesssim b^{-}(x)\mathcal{M}(\vec{f})(x).
\end{eqnarray*}
Similarly, we also have $\big|[b,\mathcal{M}]^{(2)}(\vec{f})(x)-[|b|,\mathcal{M}]^{(2)}(\vec{f})(x)\big|\lesssim b^{-}(x)\mathcal{M}(\vec{f})(x).$
Since $\big||a|-|c|\big|\leq |a-c|$ for any real numbers $a$ and $c$, there holds
$$\big|[|b|,\mathcal{M}]^{(i)}(f_{1},f_{2})(x)\big|\leq \mathcal{M}^{(i)}_{b}(f_{1},f_{2})(x),$$
for $i=1,2$. This shows that
\begin{eqnarray}\label{eq1.2.1}
\big|\Sigma\vec{b},\mathcal{M}](\vec{f})(x)\big|&\lesssim& \mathcal{M}_{\Sigma\vec{b}}(\vec{f})(x)+b^{-}(x)\mathcal{M}(\vec{f})(x).
\end{eqnarray}
Applying (\ref{eq1.2.1}) and Theorem \ref{thm1.1} we have
\begin{eqnarray*}
\big\|[\Sigma\vec{b},\mathcal{M}](\vec{f})(x)\big\|_{L^{p}(\omega^{1-p})}
&\lesssim& \big\|\mathcal{M}_{\Sigma\vec{b}}(\vec{f})\big\|_{L^{p}(\omega^{1-p})}+\big\|b^{-}\mathcal{M}(\vec{f})\big\|_{L^{p}(\omega^{1-p})}\\
&\lesssim& \big(\|b^{-}/\omega\|_{L^{\infty}}+\|b\|_{{\rm BMO}(\omega)}\big)\|f_{1}\|_{L^{p_{1}}(\omega)}\|f_{2}\|_{L^{p_{2}}(\omega)}.
\end{eqnarray*}
Therefore, $b\in {\rm BMO}(\omega)$ with $b^{-}/\omega\in L^{\infty}$ implies that $[\Sigma\vec{b},\mathcal{M}]$ is bounded from $L^{p_{1}}(\omega)\times L^{p_{2}}(\omega)$ to $L^{p}(\omega^{1-p})$.

\vspace{0.2cm}

$(B2)\Rightarrow (B1)$: Let $Q_{0}$ be any fixed cube. By Lemma \ref{lem7}, for any $x\in Q_{0}$,
$$b(x)=b(x)\mathcal{M}(\chi_{Q_{0}},\chi_{Q_{0}})(x),$$
$$M_{Q_{0}}(b)(x)=\mathcal{M}(b\chi_{Q_{0}},\chi_{Q_{0}})(x)=\mathcal{M}(\chi_{Q_{0}},b\chi_{Q_{0}})(x),$$
Then,
\begin{eqnarray*}
&&\frac{2}{\omega(Q_{0})}\int_{Q_{0}}|b(x)-M_{Q_{0}}(b)(x)|dx\\
&&=\frac{2}{\omega(Q_{0})}\int_{Q_{0}}|b(x)\mathcal{M}(\chi_{Q_{0}},\chi_{Q_{0}})(x)-\mathcal{M}(b\chi_{Q_{0}},\chi_{Q_{0}})(x)|dx\\
&&=\frac{1}{\omega(Q_{0})}\int_{Q_{0}}|b(x)\mathcal{M}(\chi_{Q_{0}},\chi_{Q_{0}})(x)-\mathcal{M}(b\chi_{Q_{0}},\chi_{Q_{0}})(x)\\
&&\qquad\qquad\qquad +b(x)\mathcal{M}(\chi_{Q_{0}},\chi_{Q_{0}})(x)-\mathcal{M}(\chi_{Q_{0}},b\chi_{Q_{0}})(x)|dx\\
&&\lesssim\frac{1}{\omega(Q_{0})}\int_{Q_{0}}\big|[\Sigma\vec{b},\mathcal{M}](\chi_{Q_{0}},\chi_{Q_{0}})(x)\big|dx\\
&&\lesssim \frac{1}{\omega(Q_{0})}\bigg(\int_{Q_{0}}\Big|[\Sigma\vec{b},\mathcal{M}](\chi_{Q_{0}},\chi_{Q_{0}})(x)\Big|^{p}\omega(x)^{1-p}dx\bigg)^{1/p}
\cdot\bigg(\int_{Q_{0}}\omega(x)dx\bigg)^{1/p'}\\
&&\lesssim\frac{1}{\omega(Q_{0})^{1/p}}\big\|[\Sigma\vec{b},\mathcal{M}](\chi_{Q_{0}},\chi_{Q_{0}})\big\|_{L^{p}(\omega^{1-p})}\\
&&\lesssim \|[\Sigma\vec{b},\mathcal{M}]\|_{L^{p_{1}}(\omega)\times L^{p_{2}}(\omega)\rightarrow L^{p}(\omega^{1-p})}.
\end{eqnarray*}

Now, we have all the ingredients to prove $b\in {\rm BMO}(\omega)$ and $b^{-}/\omega\in L^{\infty}$.
\begin{eqnarray*}
\frac{1}{\omega(Q_{0})}\int_{Q_{0}}|b(x)-b_{Q_{0}}|dx
&\lesssim&\frac{1}{\omega(Q_{0})}\int_{Q_{0}}|b(x)-M_{Q_{0}}(b)(x)|dx\\
&\lesssim&\|[\Sigma\vec{b},\mathcal{M}]\|_{L^{p_{1}}(\omega)\times L^{p_{2}}(\omega)\rightarrow L^{p}(\omega^{1-p})}.
\end{eqnarray*}
which implies that $b\in {\rm BMO}(\omega)$.

In order to show show that $b^{-}/\omega\in L^{\infty}$, observe that for any $x\in Q_{0}$, $M_{Q_{0}}(b)(x)\geq |b(x)|$. Therefore,
$$0\leq b^{-}(x)\lesssim M_{Q_{0}}(b)(x)-b^{+}(x)+b^{-}(x)=M_{Q_{0}}(b)(x)-b(x),$$
which gives
\begin{eqnarray*}
\frac{1}{|Q_{0}|}\int_{Q_{0}}\frac{b^{-}(x)}{\omega(x)}dx&\lesssim& \frac{1}{|Q_{0}|}\int_{Q_{0}}b^{-}(x)dx\cdot\frac{1}{\inf_{x\in Q_{0}}\omega(x)}\\
&\lesssim&\frac{1}{|Q_{0}|}\int_{Q_{0}}|b(x)-M_{Q_{0}}(b)(x)|dx\cdot \frac{|Q_{0}|}{\omega(Q_{0})}\\
&\lesssim&\|[\Sigma\vec{b},\mathcal{M}]\|_{L^{p_{1}}(\omega)\times L^{p_{2}}(\omega)\rightarrow L^{p}(\omega^{1-p})},
\end{eqnarray*}
this yields that
$$(b^{-}/\omega)_{Q_{0}}\lesssim \|[\Sigma\vec{b},\mathcal{M}]\|_{L^{p_{1}}(\omega)\times L^{p_{2}}(\omega)\rightarrow L^{p}(\omega^{1-p})}.$$
Thus, the boundedness of $b^{-}/\omega$ follows from Lebesgue's differentiation theorem.

\vspace{0.2cm}
$(B1)\Rightarrow (B3)$: Let $\vec{B}=(|b|,b)$ and $\vec{\mathbb{B}}=(|b|,|b|)$. Then
\begin{eqnarray*}
&&\Big|[\Pi\vec{b},\mathcal{M}](f_{1},f_{2})(x)-[\Pi\vec{B},\mathcal{M}](f_{1},f_{2})(x)\Big|\\
&&\lesssim \Big|b(x)b(x)\mathcal{M}(\vec{f})(x)-b(x)\mathcal{M}(f_{1},bf_{2})(x)\\
&&\qquad-|b(x)|b(x)\mathcal{M}(\vec{f})(x)+|b(x)|\mathcal{M}(f_{1},bf_{2})(x)\Big|\\
&&\lesssim b^{-}(x)\big|[b,\mathcal{M}]^{(2)}(f_{1},f_{2})(x)\big|.
\end{eqnarray*}
Similarly, we also have
\begin{eqnarray*}
&&\Big|[\Pi\vec{\mathbb{B}},\mathcal{M}](f_{1},f_{2})(x)-[\Pi\vec{B},\mathcal{M}](f_{1},f_{2})(x)\Big|\\
&&\lesssim \Big||b(x)||b(x)|\mathcal{M}(\vec{f})(x)-|b(x)|\mathcal{M}(bf_{1},f_{2})(x)\\
&&\qquad-|b(x)|b(x)\mathcal{M}(\vec{f})(x)+|b(x)|\mathcal{M}(bf_{1},f_{2})(x)\Big|\\
&&\lesssim b^{-}(x)\big|[|b|,\mathcal{M}]^{(1)}(f_{1},f_{2})(x)\big|.
\end{eqnarray*}

Noting that
$$\big|[\Pi\vec{\mathbb{B}},\mathcal{M}](\vec{f})(x)\big|\leq \mathcal{M}_{\Pi\vec{b}}(\vec{f})(x),$$
which yields that
\begin{eqnarray*}
\big|[\Pi\vec{b},\mathcal{M}](\vec{f})(x)\big|&\lesssim& \mathcal{M}_{\Pi\vec{b}}(\vec{f})(x)+b^{-}(x)\mathcal{M}_{\Sigma\vec{b}}(\vec{f})(x)+(b^{-}(x))^{2}\mathcal{M}(\vec{f})(x).
\end{eqnarray*}
It follows from Theorem \ref{thm1.1} and $b^{-}/\omega\in L^{\infty}$ that
\begin{eqnarray*}
&&\big\|[\Pi\vec{b},\mathcal{M}](\vec{f})(x)\big\|_{L^{p}(\omega^{1-2p})}\\
&&\lesssim \big\|\mathcal{M}_{\Pi\vec{b}}(\vec{f})\big\|_{L^{p}(\omega^{1-2p})}
+\big\|b^{-}\mathcal{M}_{\Sigma\vec{b}}(\vec{f})\big\|_{L^{p}(\omega^{1-2p})}+\big\|(b^{-})^{2}\mathcal{M}(\vec{f})\big\|_{L^{p}(\omega^{1-2p})}\\
&&\lesssim \|b\|^{2}_{{\rm BMO}(\omega)}\|f_{1}\|_{L^{p_{1}}(\omega)}\|f_{2}\|_{L^{p_{2}}(\omega)}+
\|b^{-}/\omega\|_{L^{\infty}}\big\|\mathcal{M}_{\Sigma\vec{b}}(\vec{f})\big\|_{L^{p}(\omega^{1-p})}\\
&&\qquad+\|b^{-}/\omega\|^{2}_{L^{\infty}}\|\mathcal{M}(\vec{f})\|_{L^{p}(\omega)}\\
&&\lesssim \big(\|b^{-}/\omega\|_{L^{\infty}}+\|b\|_{{\rm BMO}(\omega)}\big)^{2}\|f_{1}\|_{L^{p_{1}}(\omega)}\|f_{2}\|_{L^{p_{2}}(\omega)},
\end{eqnarray*}
this leads to our results.

\vspace{0.2cm}

$(B3)\Rightarrow (B1)$: Let $Q_{0}$ be any fixed cube. By Lemma \ref{lem5}, for any $x\in Q_{0}$,
$$b(x)^{2}=b(x)^{2}\mathcal{M}(\chi_{Q_{0}},\chi_{Q_{0}})(x),$$
$$b(x)M_{Q_{0}}(b)(x)=b(x)\mathcal{M}(b\chi_{Q_{0}},\chi_{Q_{0}})(x)=b(x)\mathcal{M}(\chi_{Q_{0}},b\chi_{Q_{0}})(x),$$
$$M^{2}_{Q_{0}}(b)(x)=\mathcal{M}(b\chi_{Q_{0}},b\chi_{Q_{0}})(x).$$
Then,
\begin{eqnarray*}
&&\frac{1}{\omega(Q_{0})}\int_{Q_{0}}|b(x)-M_{Q_{0}}(b)(x)|^{2}\omega(x)^{-1}dx\\
&&=\frac{1}{\omega(Q_{0})}\int_{Q_{0}}\Big(b(x)^{2}-2b(x)M_{Q_{0}}(b)(x)+M^{2}_{Q_{0}}(b)(x)\Big)\omega(x)^{-1}dx\\
&&=\frac{1}{\omega(Q_{0})}\int_{Q_{0}}[\Pi\vec{b},\mathcal{M}](\chi_{Q_{0}},\chi_{Q_{0}})(x)\omega(x)^{-1}dx\\
&&\lesssim \frac{1}{\omega(Q_{0})}\bigg(\int_{Q_{0}}\Big|[\Pi\vec{b},\mathcal{M}](\chi_{Q_{0}},\chi_{Q_{0}})(x)\Big|^{p}\omega(x)^{1-2p}dx\bigg)^{1/p}
\cdot\bigg(\int_{Q_{0}}\omega(x)dx\bigg)^{1/p'}\\
&&=\frac{1}{\omega(Q_{0})^{1/p}}\big\|[\Pi\vec{b},\mathcal{M}](\chi_{Q_{0}},\chi_{Q_{0}})\big\|_{L^{p}(\omega^{1-2p})}\\
&&\lesssim \|[\Pi\vec{b},\mathcal{M}]\|_{L^{p_{1}}(\omega)\times L^{p_{2}}(\omega)\rightarrow L^{p}(\omega^{1-2p})}.
\end{eqnarray*}

Now, we have all the ingredients to prove $b\in {\rm BMO}(\omega)$ and $b^{-}/\omega\in L^{\infty}$.
\begin{eqnarray*}
&&\frac{1}{\omega(Q_{0})}\int_{Q_{0}}|b(x)-b_{Q_{0}}|^{2}\omega(x)^{-1}dx\\
&&\lesssim\frac{1}{\omega(Q_{0})}\int_{Q_{0}}|b(x)-M_{Q_{0}}(b)(x)|^{2}\omega(x)^{-1}dx\\
&&\qquad+\frac{1}{\omega(Q_{0})}\int_{Q_{0}}|b_{Q_{0}}-M_{Q_{0}}(b)(x)|^{2}\omega(x)^{-1}dx\\
&&\lesssim \frac{1}{\omega(Q_{0})}\int_{Q_{0}}|b(x)-M_{Q_{0}}(b)(x)|^{2}\omega(x)^{-1}dx\\
&&\lesssim \|[\Pi\vec{b},\mathcal{M}]\|_{L^{p_{1}}(\omega)\times L^{p_{2}}(\omega)\rightarrow L^{p}(\omega^{1-2p})}.
\end{eqnarray*}
which implies that $b\in {\rm BMO}^{2}(\omega)$; that is, $b\in {\rm BMO}(\omega)$.

For any $x\in Q_{0}$, we have
\begin{eqnarray*}
\frac{1}{|Q_{0}|}\int_{Q_{0}}\frac{b^{-}(x)}{\omega(x)}dx&\lesssim& \frac{1}{|Q_{0}|}\int_{Q_{0}}b^{-}(x)dx\cdot\frac{1}{\inf_{x\in Q_{0}}\omega(x)}\\
&\lesssim&\frac{1}{|Q_{0}|}\int_{Q_{0}}|b(x)-M_{Q_{0}}(b)(x)|dx\cdot \frac{|Q_{0}|}{\omega(Q_{0})}\\
&\lesssim&\bigg(\frac{1}{\omega(Q_{0})}\int_{Q_{0}}|b(x)-M_{Q_{0}}(b)(x)|^{2}\omega(x)^{-1}dx\bigg)^{1/2}\\
&\lesssim&\|[\Pi\vec{b},\mathcal{M}]\|_{L^{p_{1}}(\omega)\times L^{p_{2}}(\omega)\rightarrow L^{p}(\omega^{1-2p})},
\end{eqnarray*}
which yields
$$(b^{-}/\omega)_{Q_{0}}\lesssim \|[\Pi\vec{b},\mathcal{M}]\|_{L^{p_{1}}(\omega)\times L^{p_{2}}(\omega)\rightarrow L^{p}(\omega^{1-2p})}.$$
Thus, the boundedness of $b^{-}/\omega$ follows from Lebesgue's differentiation theorem.

The proof of Theorem \ref{thm1.2} is complete. \qed

\section{\bf Acknowledgments}
The research was supported by National Natural Science Foundation of China (Grant NO.11661075).


\vskip 10mm

College of Mathematics and System Sciences, Xinjiang University, Urumqi 830046, Republic of China\\
\indent  Email:Wangdh1990@126.com; zhoujiang@xju.edu.cn.

\end{document}